\documentclass[reqno]{amsart}
\usepackage{latexsym,amsfonts,amssymb,epsfig}
\usepackage{hyperref} 
\newcommand{\Z}{\mathbb Z}            
\newcommand{\R}{\mathbb R}            
\newcommand{\N}{\mathbb N}            
\newcommand{\C}{\mathbb C}            

\renewcommand{\a}{\alpha}

\DeclareMathOperator{\ch}{char}

\DeclareMathOperator{\gr}{gr}

\DeclareMathOperator{\ord}{ord}
\DeclareMathOperator{\typ}{typ}

\def\lla{      \mathrel{ \mathop{\le}\limits^{\rm a}   }}

\newcommand{\CO}{\mathcal C}          
\newcommand{\WW}{\mathbf W}           
\newcommand{\CC}{\mathbf C}           
\renewcommand{\AA}{\mathbf A}         
\newcommand{\NN}{\mathbf N}           
\newcommand{\MM}{\mathbf M}           
\newcommand{\VV}{\mathbf V}           
\newcommand{\LL}{\mathbf L}           
\newcommand{\EE}{\mathbf E}           
\newcommand{\PP}{\mathbf P}           

\renewcommand {\limsup}{\operatorname* {\overline{lim}}}

\newcommand{\sh}{\mathop{\mathrm {sh}}}

\newtheorem{Th}{Theorem}[section]
\newtheorem{Lemma}[Th]{Lemma}
\newtheorem{Corr}[Th]{Corollary}
\newtheorem{Remark}{Remark}

\begin{document}
\title{Scale for codimension growth of Poisson PI-algebras}
\author{Victor Petrogradsky}
\address{Department of Mathematics, University of Brasilia, 70910-900 Brasilia DF, Brazil}
\email{petrogradsky@rambler.ru}
\subjclass[2000]{16R10, 17B63, 17B50, 17B01, 17B30, 17B65}
\keywords{PI-algebras; codimension growth; varieties of linear algebras;
identical relations; Poisson algebras;  solvable Lie algebras}
\begin{abstract}
A.Regev proved that the codimension growth of an associative PI-algebra is at most exponential.
The author established a scale for the codimension growth of Lie PI-algebras,
which includes a series of functions between exponential and factorial functions.
We prove that the same scale stratifies the ordinary codimension growth of Poisson PI-algebras satisfying Lie identical relations.
As a byproduct, we obtain a new bound on the codimension growth of Lie PI-algebras in terms of complexity functions.
We also study very fast codimension growth of some Poisson algebras without Lie identities.
We essentially use techniques of exponential generating functions and growth of respective fast growing entire functions.
\end{abstract}

\maketitle
An important instrument to study structure and properties of associative algebras is the theory of identical relations.
Now, this is the established area of the modern algebra~\cite{Drensky,GiaZai}.
The theory of identical relations in Lie algebras
has many applications to group theory such as the solution of the Restricted Burnside Problem~\cite{Ba,Vau}.
Also, the identical relations are studied in another algebraic structures.

Poisson algebras were introduced in 1976 by Berezin~\cite{Ber67}, see also Vergne~\cite{Ver69} (1969).
Poisson algebras naturally appear in different areas of algebra, topology and physics.
Using Poisson algebras, Shestakov and Umirbaev solved a long-standing problem
on the Nagata automorphism of the polynomial ring in three variables $\mathbb{C}[x,y,z]$, they proved that it is wild~\cite{SheUmi04}.
Different algebraic properties of free Poisson algebras were studied by
Makar-Limanov, Shestakov and Umirbaev~\cite{MakShe12,MakUmi11}.

We fix basic notations. By $K$ denote the ground field of an arbitrary characteristic.
By $\langle S\rangle$ or $\langle S\rangle_K$ denote the linear span of a subset $S$ in a $K$-vector space.
Let $L$ be a Lie algebra.
The Lie brackets are left-normed: $[x_1,\ldots,x_n]=[[x_1,\dots,x_{n-1}],x_n]$, $n\ge 1$.
By $U(L)$ denote the universal enveloping algebra.

\section{Introduction: Codimension growth of linear PI-algebras and Poisson algebras}

\subsection{Codimension growth of associative PI-algebras}
Let $A$ be an (associative, or another type) algebra satisfying a nontrivial identical relation, then $A$ is called a {\it PI-algebra}.
Let $P_n(X_1,\ldots,X_n)$ be the span of all multilinear monomials in $X_1,\ldots,X_n$
in the free (associative) algebra $A$ generated by $X=\{X_i|i\in \N\}$.
Denote by $W_n\subset P_n$ all elements that are identical relations in $A$.
A.Regev introduced so called {\it codimension sequence} $\{c_n(A):=\dim P_n/W_n\mid n\ge 0\}$.
\begin{Th}[Regev \cite{Regev72}] Let $A$ be an associative PI-algebra.
Then there exists $a>1$ such that $c_n(A)<a^n$ for $n\in\N$.
\end{Th}
This sequence proved to be useful to study associative algebras.
So, A.Regev proved that the tensor product of associative PI-algebras is a PI-algebra~\cite{Regev72}.
A.~Giambruno and M.~Zaicev proved that a "precise" exponent
for an associative PI-algebra in characteristic zero always exists and is an integer~\cite{GiaZa99},
see also the monograph~\cite{GiaZai}.

\subsection{Codimension growth of Lie PI-algebras}
Let $\VV$ be a variety of Lie algebras, it is immediate that $c_n(\VV)\le (n-1)!$ for $n\ge 1$.
\begin{Th} [A.Grishkov \cite{Grishkov88}]
Let $L$ be a Lie algebra satisfying a nontrivial identity.
Then for any $r>1$ there exists $N_0$ such that
$$ c_n(L)\le \frac {n!}{r^n},\quad n\ge N_0.$$
\end{Th}

Yu.P.Razmyslov suggested to consider an exponential generating function, called {\it complexity function}
as $\CO(L,z)=\sum_{n\ge 1} \frac {c_n(L)}{n!}z^n$.
The result above was equivalently rewritten as:
\begin{Th} [Yu.P.Razmyslov~\cite{Razmyslov}]\label{Traz}
Let $L$ be a Lie algebra satisfying a nontrivial identity.
Then $\CO(L,z)$ is an entire function of complex variable.
\end{Th}

Recall standard definitions of varieties of Lie algebras~\cite{Ba}.
One defines the {\it lower central series}: $\gamma_1(L)=L$, $\gamma_{n+1}(L)=[\gamma_{n}(L), L]$, $n\ge 1$.
In particular, $L^2=[L,L]=\gamma_2(L)$ is the {\it commutator subalgebra}.
A Lie algebra is {\it nilpotent} of class $s$ provided that $\gamma_{s+1}(L)=0$ while $\gamma_{s}(L)\ne 0$.
All Lie algebras nilpotent of class at most $s$ form a variety denoted by $\NN_s$.

A Lie algebra $L$ is {\it solvable} of length $q$ provided that there exists
a chain of ideals $0=L_{q+1}\subset L_q\subset\dots\subset L_2\subset L_1=L$
such that $L_i^2\subset L_{i+1}$, $i=1,\dots,q$, and $q$ is minimal with these conditions.
All Lie algebras solvable of length at most $q$ is a variety denoted by $\AA^q$.

More generally,  a Lie algebra $L$ is {\em polynilpotent} with a fixed tuple of natural numbers
$(s_q,\dots,s_2,s_1)$~\cite{Ba} provided that there exists  a chain of ideals
$0=L_{q+1}\subset L_q\subset\dots\subset L_2\subset L_1=L$, where  $L_i/L_{i+1}$ is nilpotent of class $s_i$, $i=1,\dots,q$.
Denote by $\NN_{s_q}\cdots \NN_{s_2}\NN_{s_1}$  the class of all such Lie algebras.
As the particular case $s_q=\cdots=s_1=1$ we get the variety  ${\AA}^q$ of solvable Lie algebras of length~$q$.
Let us make the following observation. First, any polynilpotent Lie algebra is solvable.
Second, the {\em free polynilpotent} Lie algebras (the tuple being fixed) are particular interesting examples of solvable Lie algebras.

The author constructed the following hierarchy for the codimension growth of Lie algebras that consists
of a series of subfactorial functions~\cite{Pe95}:
\begin{equation}\label{scale}
\Psi^q_\a(n)=
\left\{
\begin{array}{lll}
\displaystyle
\big(n!\big)^{1-1/\a},\quad & \a\ge 1,\quad & q=2;\\
\displaystyle
\frac{n!}{(\ln^{(q-2)}n)^{n/\a}},\quad & \a>0,\quad & q=3,4,\dots
\end{array}
\right.
\end{equation}
The origin of these functions is due to the following result.
It is derived from the study of the respective complexity functions and their growth
(see Theorem~\ref{TordNN}).

\begin{Th}[{\cite{Pe95,Pe97msb}}]
Let $\AA^q$ be the variety of solvable Lie algebras of length $q\ge 3$. Then
$$ c_n(\AA^q)=
  \frac{n!}{(\ln^{(q-2)}n)^{n}}
                  \left(1+o(1)\right)^{n},\qquad n\to\infty.$$
\end{Th}
Recall that the class of solvable Lie algebras $\AA^q$ of length $q$
is a particular case of the variety of polynilpotent Lie algebras $\NN_{s_q}\cdots\NN_{s_1}$.
A more general statement is as follows.
\begin{Th}[{\cite{Pe97msb}}] \label{Tpoly}
Consider a tuple of integers $(s_q,\ldots,s_1)$, where $q\ge 2$, and the respective variety of polynilpotent Lie algebras
$\VV=\NN_{s_q}\cdots\NN_{s_1}$. Then
$$ c_n(\VV)=
  \left\{
  \begin{array}{ll}
  \big(n!\big)^{1-1/s_1} (s_2+o(1))^{n/s_1};\     &q=2;\\
  \displaystyle
  \frac{n!}{(\ln^{(q-2)}n)^{n/s_1}}
                  \left(\frac{s_2+o(1)}{s_1}
                  \right)^{n/s_1}\!\!\!\!\!;\qquad
                                         &q=3,4,\dots,
  \end{array}
  \right.\quad n\to\infty.$$
\end{Th}
Our scale is complete in the sense that the codimension growth of a nontrivial variety
of Lie algebras is bounded by some of these functions.

\begin{Th}[{\cite{Pe97msb}}]\label{Tcodimupper}
Let $L$ be a Lie algebra satisfying a nontrivial identity of degree $q>3$.
Then there exists an infinitesimal such that
$$ c_n(L)\le \frac{n!}{(\ln^{(q-3)}n)^n}(1+o(1))^n,\qquad n\to\infty. $$
\end{Th}

The codimension growth was studied also for Jordan algebras~\cite{Drensky87,GiaZe11},
absolutely free (commutative and anticommutative) algebras~\cite{Pe05}, and arbitrary linear algebras~\cite{GiaMiZai08}.

\subsection{Poisson algebras}

A $K$-vector space $A$ is called a {\it Poisson algebra}
if $A$ is supplied with two $K$-bilinear operations:
\begin{itemize}
\item
$A$ is a commutative associative algebra with unit and multiplication $a\cdot b$ (or $ab$), where $a, b\in A$;
\item
$A$ is a Lie algebra which product is traditionally denoted by the {\it Poisson bracket} $\{a, b\}$, where $a, b\in A$;
\item these two operations are related by the Leibnitz rule:
\begin{equation*}
\{a\cdot b, c\}=a\cdot\{b, c\}+b\cdot\{a, c\},\qquad  a, b, c \in A.
\end{equation*}
\end{itemize}


Let $L$ be a Lie algebra over an arbitrary field $K$, and
$\{U_n| n\ge 0\}$ the natural filtration of its universal enveloping algebra $U(L)$.
Consider the {\it symmetric algebra}
$S(L):=\gr U(L)=\mathop{\oplus}\limits_{n=0}^\infty U_{n}/U_{n-1}$.
Recall that $S(L)$ is identified with the polynomial ring $K[v_i\,|\, i\in I]$,
where $\{v_i\,|\, i\in I\}$ is an arbitrary $K$-basis of $L$ (see e.g.~\cite{Ba,Dixmier,BMPZ}).
Set $\{v_i,v_j\}=[v_i,v_j]$ for all $i,j\in I$.
Extending to the whole of $S(L)$ by linearity and using the Leibnitz rule, one obtains the Poisson bracket:
$$\{v_i\cdot v_j,v_k\}=v_i\cdot\{v_j,v_k\}+v_j\cdot\{v_i,v_k\},\qquad i,j,k\in I.$$
Now, $S(L)$ has a structure of a Poisson algebra, called the {\it symmetric algebra} of $L$.

\subsection{Free Poisson algebras}
Consider the free Lie algebra $L=L(X)$ generated by a set $X$ and its symmetric algebra $F(X)=S(L(X))$.
Then, $F(X)$ is a {\it free Poisson algebra} in $X$, as was shown by I.Shestakov~\cite{Shestakov93}.
For example, let $L=L(x,y)$ be the free Lie algebra of rank 2. Consider its Hall basis~\cite{Ba}
$$L=\langle x, y, [y,x], [[y,x],x], [[y,x],y], [[[y,x],x],x],\ldots\rangle_K.$$
We obtain the free Poisson algebra $F(x,y)=S(L)$ of rank 2, it has the following canonical basis:
\begin{equation}\label{basisP}
F(x,y)=\big\langle x^{n_1} y^{n_2} \{y,x\}^{n_3} \{\!\{y,x\},x\}^{n_4} \{\!\{y,x\},y\}^{n_5}  \{\!\{\! \{y,x\},x\},x\}^{n_6}\cdots\,\big|\, n_i\ge 0
\big\rangle_K,
\end{equation}
the number of non-zero powers $n_i$  above being finite.

\subsection{Poisson identities, customary identities}
A definition of  a {\it Poisson PI-algebra} is standard,
identities being elements of the free Poisson algebra  $F(X)$ of countable rank.
Assume that basic facts on identical relations of linear algebras are known  (see, e.g.,~\cite{Ba,Drensky,GiaZai}).

A basic theory of identical relations for Poisson algebras  was developed by D.Farkas~\cite{Farkas98,Farkas99}.
He suggested to use {\it customary identities}~\cite{Farkas98}:
\begin{equation}\label{customary}
\sum_{\substack{\sigma\in S_{2n}\\ \sigma(2k-1)<\sigma(2k),\ k=1,\dots,n\\\sigma(1)<\sigma(3)<\cdots<\sigma(2n-1)}}
\!\!\!\!\!\!\!\!\!\!\!\!\!
\mu_{\sigma} \{x_{\sigma (1)},x_{\sigma (2)}\}\cdots
\{x_{\sigma (2n-1)},x_{\sigma (2n)}\}\equiv 0,\quad \mu_\sigma\in K; \ \mu_e=1.
\end{equation}
He proved that a nontrivial multilinear Poisson identity implies a customary identity.
The standard identity was defined as a particular case of customary identities~\cite{Farkas98, Farkas99}.
The related {\it customary codimension growth} was introduced and applied  in characteristic zero
by S.Mishchenko, V.Petrogradsky, and A.Regev~\cite{MiPeRe}.

Algebras and varieties of Poisson algebras of slow ordinary codimension growth,
such as polynomial, (almost) exponential, minimal with respect to some condition,
were in details studied by S.Ratseev~\cite{Rats13,Rats14,Rats16,RatsCherv16}.

Identical relations of symmetric Poisson algebras
of Lie (super)algebras initially were studied by Kostant~\cite{Kos81}, I.Shestakov~\cite{Shestakov93}, and D.Farkas~\cite{Farkas99}.
A.Giambruno and V.Petrogradsky found necessary and sufficient conditions
for existence of non-trivial multilinear Poisson identical relations in truncated symmetric algebras
of (restricted) Lie algebras~\cite{GiPe06}.
So called pure Lie identical relations in (truncated) symmetric Poisson algebras,
such as solvability and Lie nilpotence were studied further in~\cite{PeIlana17,Sic20,SicUse21Bull}.

\section{Main results: Scale for codimension growth of Poisson algebras}

We distinguish {\bf two types} of elements of the free Poisson algebra and identities they yield:
\begin{itemize}
\item
The elements of the free Lie algebra $L(X)$, which are naturally contained in $F(X)=S(L(X))$,
we call {\it (pure) Lie identities}, e.g. such as
the identity of Lie nilpotence of class $s$:\\
$\{\{X_1,X_2\},\ldots X_{s+1}\}\equiv 0$.
\item The remaining elements are called {\it mixed identities}. For example, consider:
\begin{equation}\label{x1x2}
\{X_1,X_2\}\cdot \{X_3,X_4\}\equiv 0.
\end{equation}
\end{itemize}

Our goal is to show that the scale for the codimension growth for Lie algebras~\eqref{scale}
also stratifies the ordinary codimension growth of Poisson PI-algebras of the first type, e.g. {\bf satisfying a non-trivial Lie identity}.

Suppose that $\VV$ is a variety of Lie algebras. Then by $\PP\VV$ we denote the variety of Poisson algebras
defined by the Lie identical relations of $\VV$.
So, we obtain the variety of solvable Poisson algebras $\PP\AA^q$ and the
variety of Lie polynilpotent Poisson algebras $\PP\NN_{s_q}\cdots\NN_{s_1}$.

Let $P_n=P_n(X_1,\ldots,X_n)$ be the span of all multilinear monomials in $X_1,\ldots,X_n$
in the free Poisson algebra $F(X)=S(L(X))$ generated by $X=\{X_i|i\in \N\}$,
i.e. we consider multilinear products of Lie commutators, for example, using the basis of the free Poisson algebra~\eqref{basisP}.
Let $F(\VV,X)$ be the relatively free algebra of a variety of Poisson algebras $\VV$ and
$\phi:F(X)\to F(\VV,X)$ the natural epimorphism.
We define the {\it ordinary codimension sequence} $c_n(\VV):=\dim \phi(P_n)$, $n\ge 0$.
We warn that in case of Poisson algebras, the {\it customary codimension growth} is also studied,
defined as the dimension of the space of multilinear customary polynomials~\eqref{customary}
in a relatively free algebra, see~\cite{Farkas98,Farkas99,MiPeRe}.
In this paper we study the ordinary codimension growth only, the term "ordinary" can be omitted.

We formulate the first main result of the paper.
\begin{Th}\label{TupperP}
Let $\WW$ be a variety of Poisson algebras satisfying a nontrivial {\bf Lie identity} of degree $m\ge 3$ and
$\CO(\WW,z)=\sum\limits_{n=0}^\infty \frac {c_n(\WW)}{n!} z^n$ its complexity function, the field being arbitrary. Then
\begin{enumerate}
\item
The following coefficientwise bound   holds:
$$
\CO(\WW,z)\preceq \underbrace{\exp(z\exp(z \exp(\cdots(z\exp(z\exp}\limits_{m-1 \text{\rm{ times }} \exp } z)) \cdots))).
$$
\item There exists an infinitesimal such that
$$ c_n(\WW) \le \frac{n!} {(\ln^{(m-2)}n)^n}(1+o(1))^n, \qquad n\to\infty. $$
\end{enumerate}
\end{Th}

The second main result specifies the place of solvable Poisson algebras  $\PP\AA^q$ on our scale.
We prove our result in generality of free Lie polynilpotent Poisson algebras as follows.

\begin{Th}\label{TPpoly}
Consider a tuple of integers $(s_q,\ldots,s_1)$, where $q\ge 3$, and the variety of Lie polynilpotent Poisson algebras
$\PP\NN_{s_q}\cdots\NN_{s_1}$, a field $K$ being arbitrary.
Then the ordinary codimension growth has the same parameters as the respective variety of polynilpotent Lie algebras
$\NN_{s_q}\cdots\NN_{s_1}$ (see Theorem~\ref{Tpoly}):
$$ c_n(\PP\NN_{s_q}\cdots\NN_{s_1})=
  \displaystyle
  \frac{n!}{(\ln^{(q-2)}n)^{n/s_1}}
                  \left(\frac{s_2+o(1)}{s_1}
                  \right)^{n/s_1},\qquad  n\to\infty.
$$
\end{Th}
Our approach is based on study of exponential generating functions of algebras and sets
and the technique of growth of fast growing entire functions developed by the author in~\cite{Pe97msb,Pe99JMSciSch_exp,Pe99Isr},
all necessary results are given in Section~\ref{Scomp}.

As an intermediate result, we obtain a new bound on the codimension growth of Lie algebras
in terms of complexity functions (Theorem~\ref{TboundL}).
To prove the second theorem we use a resent result of S.Siciliano and H.Usefi~\cite{SicUse21} (Theorem~\ref{TSicUse}),
due to which we turn to center-by-metabelian Lie algebras, which were extensively studied before.
Now we need to know the respective complexity function,
for which we find explicit formulas  in case of an arbitrary field, see Theorem~\ref{Tcenterby}.
Finally, we show that the ordinary codimension growth of Poisson algebras without
Lie identical relations can be very high (Theorem~\ref{TWs}).

\begin{Remark}
We expect the results will be applied to study overexponential codimension growth of varieties of Jordan algebras started in~\cite{Drensky87,GiaZe11}.
We conjecture that scale~\eqref{scale} (or similar to it) will appear in case of Jordan PI-algebras as well.
\end{Remark}

\section{Complexity functions and fast growing entire functions}
\label{Scomp}

In this section, we present facts on exponential generating functions of varieties of linear algebras,
called {\it complexity functions} and growth of related entire functions.

\subsection{Codimension growth of varieties of linear algebras and complexity functions}
Let $F(\VV,X)$ be the relatively free algebra of a variety of linear algebras $\VV$ generated by countable set $X=\{ x_i| i\in\N\}$.
Let $P_n \subset F(\VV,X)$ be the subspace of multilinear polynomials of degree $n$ in
variables $\{x_1,\dots,x_n\}$. One defines so called
{\em codimension growth sequence}: $c_n(\VV)=\dim P_n$, where $n\in\N$.
It was introduced by A.~Regev in order to study identical relations of associative algebras~\cite{Regev72}.
After Yu.P.Razmyslov~\cite{Razmyslov}, we consider the
{\em complexity function}.
This is the exponential generating function~\cite{GouJac} of the sequence $c_n$
$$\CO(\VV,z)=\sum_{n=0}^\infty \frac {c_n(\VV)}{n!}\, z^n,\quad z\in\C.$$

Let us give a more general setting in case of sets of monomials,
which are naturally extended to the case of vector spaces and algebras.
Let $A$ be a set of monomials in letters $X=\{x_i\,|\,i\in\mathbb N \}$.
Let $\widetilde{X}=\{x_{i_1},\dots,x_{i_n}\}\subset X$;
by  $P_n(A,\widetilde{X})$ denote the set
of all multilinear  elements of degree $n$ in $\widetilde{X}$ belonging to $A$.
Suppose that the number of elements $c_n(A,\widetilde{X})$
does not depend on the choice of $\widetilde{X}$, but depends only on $n$.
In this case we write $c_n(A)=c_n(A,\widetilde{X})$ and say that
$A$ is $X$-{\em uniform} and define the
{\it complexity function with respect to} $X$:
$${\mathcal C}_X(A,z)=\sum_{n=1}^\infty \frac{c_n(A)}{n!}z^n,\quad z\in \mathbb C.$$

{\it Examples}.
\begin{enumerate}
  \item
$\AA$ --- variety of all associative algebras
\begin{align}
P_n(\AA,\{x_1,\dots,x_n\})&=\langle x_{\pi(1)}\cdots x_{\pi(n)}\mid \pi\in S_n\rangle,
\quad c_n(\AA)=n!,\quad  n\ge 0,\nonumber\\
\CO(\AA,z)&=\sum_{n=0}^\infty \frac {n!}{n!}z^n=\frac 1{1-z}.
\label{commutative}
\end{align}
  \item
$\MM$ --- variety of commutative associative algebras ($XY-YX\equiv 0$)
\begin{align*}
P_n(\MM,\{x_1,\dots,x_n\})&=\langle x_1\cdots x_n\rangle,\quad  c_n(\MM)=1,\quad n\ge 0,\\
\CO(\MM,z)&=\sum_{n=0}^\infty \frac 1{n!}z^n=\exp(z).
\end{align*}
  \item
$\LL$ --- variety of all Lie algebras
\begin{align}
\nonumber
P_n(\LL,\{x_1,\dots,x_n\})&=\langle [x_n,x_{\pi(1)},\ldots,x_{\pi(n-1)}]\rangle,\quad c_n(\LL)=(n-1)!,\quad n\ge 1,\\
\CO(\LL,z)&=\sum_{n=1}^\infty \frac 1{n}z^n=-\ln(1-z).
\label{complL}
\end{align}
  \item
Let $\PP$ be the variety of all Poisson algebras over an arbitrary field.
By construction of the free Poisson algebra $F(X)\cong S(A(X))$, we get
\begin{equation}\label{cofrrP}
c_n(\PP)=n!,\quad  n\ge 0;\qquad
\CO(\PP,z)=\sum_{n=0}^\infty \frac {n!}{n!}z^n=\frac 1{1-z}.
\end{equation}
\end{enumerate}

The language of exponential generating functions is natural to study
codimension growth because of the following facts.
\begin{Lemma}[{\cite{GouJac,Pe97msb}}]\label{Lprod}
Consider sets of monomials $A$, $B$ in $X=\{x_i| i\in \N\}$,
let there exist complexity functions:
$$\CO(A,z)=\sum_{n=0}^\infty\frac{c_n(A)}{n!}z^n,\quad
  \CO(B,z)=\sum_{n=0}^\infty\frac{c_n(B)}{n!}z^n.$$
Then for the formal concatenation product of monomials $A\cdot B$
we have $\CO(A\cdot B,z)=\CO(A,z)\CO(B,z)$.
\end{Lemma}

\begin{Th}[{\cite{GouJac,Razmyslov,Pe97msb,Pe11}}]\label{Texp}
Let a Lie algebra $L$ be uniform with respect to the generating set $X=\{x_i| i\in \N\}$.
Then the complexity function of its universal enveloping algebra satisfies:
$$\CO(U(L),z)=\exp \CO(L,z). $$
\end{Th}

\subsection{Growth of fast growing entire functions, applications}
We start with observations on growth of fast growing entire functions. Fix notations
\begin{align*}
&\ln^{(1)}x:=\ln x,\quad\quad\  \ln^{(q+1)}x :=\ln(\ln^{(q)}x);\\
&\exp^{(1)}x:=\exp x,\quad \exp^{(q+1)}x:=\exp(\exp^{(q)}x),\quad q\in \N.
\end{align*}

Let  $f(z)$   be an entire function of complex variable.
We use standard notation ${\rm M}_f(r):=\max\limits_{|z|=r}|f(z)|$ for $r\in\R^+$.
Fix an integer $q\ge 1$.
We define the {\em order} and  {\em type} of a {\em level}
$q$ (the latter we define provided that $0<\rho=\ord_q f<\infty$) introduced by the author in~\cite{Pe97msb}:
\begin{align*}
\ord_q f &: = \inf\{\rho \mid {\rm M}_f(r)\lla \exp^{(q)}(r^\rho) \}
    =\limsup_{r\to+\infty} \frac{\ln^{(q+1)} {\rm M}_f(r)} {\ln r};\\
\typ_q f &:= \inf\{\sigma\mid {\rm M}_f(r)\lla \exp^{(q)}(\sigma r^\rho) \}=
 \limsup_{r \to+\infty}\frac{\ln^{(q)}{\rm M}_f(r)}{r^\rho}.
\end{align*}
Let $f(z)$ be an entire function and for some $q\in\N$  we have $0<\rho=\ord_q f<\infty$ and  $0<\sigma=\typ_q f<\infty$.
Then ${\rm M}_f(r)$ behaves like $\exp^{(q+1)}(\sigma r^\rho)$ as $r=|z|\to +\infty$.

In the particular case $q=1$, we get the classical notions of the order $\ord f$ and type $\typ f$ of an entire function~\cite{Evg}.
Let $\ord f<\infty$, then $f(z)$ is said of {\it finite order}.
Classical Adamar's formula connects the growth of an entire function of finite order with an asymptotic of its Taylor coefficients~\cite[3.2.3]{Evg}.

The author studied the complexity functions of solvable Lie algebras, they are entire functions of infinite order as a rule~\cite{Pe97msb}.
In order to describe respective asymptotic of the codimension growth,
the author established  the following analogue of Adamard's formula for entire functions of infinite order.

\begin{Th}[{\cite[Theorem 4.2]{Pe97msb}}]
\label{TSher}
  Suppose that $f(z)=\sum\limits_{n=0}^\infty a_n z^n$  is an entire function.
  Then for any fixed numbers $q\in {\Bbb N}$, $q\ge 3$; $\lambda > 0$ holds
  $$ \limsup_{r\to+\infty} {\frac {\ln^{(q-1)} {\rm M}_f(r)} {r^\lambda}}= \limsup_{n\to\infty} |a_n|^{\lambda/n} \ln^{(q-2)}n .
  $$
\end{Th}
Denote both sides above by $\sigma$, assume that $0<\sigma<\infty$.
Then roughly speaking, ${\rm M}_f(r)$ behaves like $\exp^{(q-1)}(\sigma r^\lambda)$ as $|z|=r\to+\infty$.
Formally, $\ord_{q-1} f=\lambda$ and $\typ_{q-1} f=\sigma$.
And this is equivalent to the fact that $|a_n|$ behave like $\big(\frac \sigma{\ln^{(q-2)}n}\big)^{n/\lambda}$ as $n\to\infty$.

The crucial step in establishing asymptotic for the codimension growth of solvable Lie algebras
is the following fact on the growth of the respective complexity function~\cite{Pe97msb}:
\begin{Th}[{\cite[Theorem 3.2, Corollary 4.1.]{Pe97msb}}]
\label{TordNN}
  Let $\VV=\NN_{s_q}\cdots\NN_{s_1}$, $q\ge 2$ be
  a polynilpotent variety of Lie algebras and
  $f(z)=\CO(\VV,z)$ its complexity function. Then
\begin{enumerate}
\item
  $\displaystyle
  \lim_{r\to+\infty}
  \frac{\ln^{(q-1)}{\rm M}_f(r)}{r^{s_1}}=\frac{s_2}{s_1}.
  $
\item in terms of our notions of order and type of high level, this fact is equivalent to
$$\ord_{q-1}f(z)=s_1,\qquad  \typ_{q-1}f(z)=s_2/s_1.$$
\end{enumerate}
\end{Th}

\section{Upper bounds on complexity functions and codimension growth}

The goal of this section is to establish an upper bound
on the complexity function and the codimension growth of a variety of Poisson algebras satisfying
a non-trivial Lie identical relation (Theorem~\ref{TupperP}).
We also obtain a new more clear bound on the complexity function of a Lie PI-algebra (Theorem~\ref{TboundL}).
\subsection{Lie indecomposable words, sets $Q_m$, $R_m$}
Let $X$ be a non-empty set and $X^*$ the set of all words in $X$.
We order $X^*$ lexicographically from the left.
A word $w\in X^*$ is called {\it regular} if for any decomposition $w=ab$ we have $w>ba$.
Let $A(X)$ be the free associative algebra generated by $X$ and
$L(X)\subset A(X)$ the free Lie algebra~\cite{Ba}.
For any $a\in A(X)$ denote by $\hat a$ the senior word in its expansion, first with respect to degree and then lexicographically.
Let $w$ be a regular word, and $[w]\in L(X)$ an arrangement of Lie brackets on it.
An arrangement of brackets is {\it acceptable} provided that $\widehat {[w]}=w$.
Each regular word has an acceptable arrangement of brackets (not necessarily unique).
By fixing for each regular word an acceptable arrangement of brackets, the respective Lie monomials yield
the Lyndon-Shirshov basis for the free Lie algebra $L(X)$~\cite{Ba,BMPZ}.

Let $X=\{x_i| i\in \N\}$.
From now on we consider multilinear words only.
Denote $P_n(X)=\{x_{\pi(1)}\cdots x_{\pi(n)}\mid \pi\in S_n\}\subset X^*\subset A(X)$.
Observe that a word $w\in P_n(X)$ is regular if and only if $w$ starts with the senior letter $x_n$.

Fix an integer $m\ge 1$.
A multilinear word $w$ is called $m$-{\it Lie decomposable} provided that
\begin{equation}\label{m-decomp}
w=a\cdot w_m w_{m-1}\cdots w_1\cdot b,\qquad w_m>w_{m-1}>\cdots> w_1,
\end{equation}
where $w_i$ is regular, i.e. it starts with its senior letter $y_i$, for all $i=1,\ldots,m$.
Remark that the ordering of words above is equivalent to the ordering of their first letters: $y_m>\cdots> y_1$.
If such a presentation does not exist, $w$ is called $m$-{\it Lie indecomposable}.
Denote by $Q_m$ the set of all such words.
Denote by $R_m\subset Q_m$ the subset of all {\it regular $m$-Lie indecomposable} words.
By definition, $1\in Q_m$ and $1\notin R_m$ for $m\ge 1$.
Assume that $w$ above~\eqref{m-decomp} is regular.
Since a regular multilinear word starts with the maximal letter, then $a$ is also regular and
 $a>w_m>\cdots>w_2$ yields another decomposition.
Thus, in case of a regular word $w$, we may assume that $a=1$ in decomposition~\eqref{m-decomp}.

\begin{Lemma}[{\cite[1.2.1]{Razmyslov}}] \label{Lw1}
Consider a multilinear word $w=w'x_nw''$, where $x_n$ is the maximal letter in it.
Then $w\in Q_m$ if and only if $w'\in Q_m$ and $w''\in Q_{m-1}$.
\end{Lemma}

\begin{Lemma} \label{Lw2}
The sets of ordinary and regular $m$-Lie indecomposable words $Q_m$ and $R_m$
are described by the following initial conditions and recurrence relations:
   \begin{enumerate}
   \item $R_1=\emptyset$, $Q_1=\{1\}$.
   \item $R_m=\{xw_0 \mid  x\in X,\ w_0\in Q_{m-1},\ x>w_0\}$, $m>1 $.
   \item
   $Q_m=\{v_1v_2\cdots v_s \mid v_i\in R_m,\ v_1<\cdots< v_s,\ s\ge 0\}$, $m>1$.
   \end{enumerate}
\end{Lemma}
\begin{proof}
Claim 1 is valid by definitions.
Claim 2. Recall that a regular multilinear word starts with its maximal letter and use Lemma~\ref{Lw1}.

Claim 3. Let $1\ne w\in Q_m $ with the maximal letter $x_n$. By Lemma~\ref{Lw1},
$w=w'x_n u_1$, $w'\in Q_m$ and $u_1\in Q_{m-1}$.
If $w'$ is nonempty, we repeat the process to $w'$ and so on. As a result, we get
\begin{align*} 
 &w=x_{i_s}u_s\cdots x_{i_2}u_2 x_{i_1}u_1,\qquad u_j\in Q_{m-1},\
 j=1,\dots,s;\qquad s\ge 1;\\
 &i_s<\cdots<i_2<i_1=n;\qquad x_{i_t}\mbox{ -- the maximal letter in }
 x_{i_s}u_s\cdots x_{i_t}u_t.
\end{align*}
By the first claim $v_t:=x_{i_t}u_t\in R_m$,
we get $w=v_s\cdots v_1$, where $v_s<\cdots<v_1$, this is the desired decomposition.
Conversely, any word $w$ above belongs to $Q_m$ by Lemma~\ref{Lw1}.
\end{proof}

Let $m\in \N$.
Introduce sequences $a_m(n):=c_n(Q_m)$, $b_m(n):=c_n(R_m)$, $n\ge 0$, and respective exponential generating functions:
\begin{align*} \label{Gqmrmz}
q_m(z)&:=\CO(Q_m,z)=\sum_{n=0}^{\infty}\frac{a_m(n)}{n!}z^n, \qquad m\in \N;\\
r_m(z)&:=\CO(R_m,z)=\sum_{n=1}^{\infty}\frac{b_m(n)}{n!}z^n, \qquad m\in \N.
\end{align*}
\begin{Lemma} \label{Lrq}
The exponential generating functions for the sets of ordinary and regular $m$-Lie indecomposable words $Q_m$ and $R_m$
satisfy the following initial conditions and recurrence relations:
   \begin{enumerate}
    \item $r_1(z)=0$, $q_1(z)=1$;
    \item $  r_m(z)=\int_0^z q_{m-1}(z)\, dz$,\quad  $m>1$;
    \item $ q_m(z)=\exp(r_m(z))=
            \exp\left(\int_0^z q_{m-1}(z)\,dz\right)$,\quad  $m>1$.
   \end{enumerate}
\end{Lemma}
\begin{proof}
1) By definition, any non-empty word is $1$-Lie decomposable,
hence $a_1(n)=0$, $n>0$, and by definition  $a_1(0)=1$. We get $q_1(z)=1$ and $r_1(z)=0$.

2) By Lemma~\ref{Lw2}, for any $w\in R_m\cap P_n(X^*)$ we have
$w=x_nu$, where $u\in Q_{m-1}\cap P_{n-1}(X^*)$, hence $c_n(R_m)=c_{n-1}(Q_{m-1})$.
Since $r_m(0)=b_m(0)=0$, in terms of generating functions we get
$r_m(z)=\int_0^z q_{m-1}(z)\, dz$.

3) The ordered products in the third claim of Lemma~\ref{Lw2} are the same as in PBW-theorem.
Applying Theorem~\ref{Texp},  we get
$\CO(Q_m,z)=\exp(\CO(R_m,z))=\exp(r_m(z))$.
\end{proof}

\subsection{Sets $\tilde Q_m$, $\tilde R_m$  and their generating functions}
Let $f(z)=\sum_{n=0}^\infty a_n z^n $ and $g(z)=\sum_{n=0}^\infty b_n z^n $ be series with real coefficients.
We write $f(z)\preceq g(z)$ to denote that $a_n\le b_n$ for all $n\ge 0$.

Define a sequence of multilinear words in $X$
using the following initial conditions and recurrence relations:
\begin{equation}\label{tildeRQ}
\begin{split}
   \tilde R_1&:=\emptyset,\quad  \tilde Q_1:=\{1\};\\
   \tilde R_m&:=\{xw_0 \mid  x\in X,\ w_0\in \tilde Q_{m-1}\},\qquad  m>1;\\
   \tilde Q_m&:=\{v_1v_2\cdots v_s \mid v_i\in \tilde R_m,\ v_1<\cdots< v_s,\ s\ge 0\},\qquad m>1.
\end{split}
\end{equation}
The difference with recurrences of Lemma~\ref{Lw1} is that for $\tilde R_m$ we omitted the condition that $x>w_0$.
Next, introduce related integer sequences and exponential generating functions:
\begin{align*} 
\tilde a_m(n)&:=c_n(\tilde Q_m),\qquad \tilde b_m(n):=c_n(\tilde R_m),\qquad n\ge 0,\ m\in\N;\\
\tilde q_m(z)&:=\CO(\tilde Q_m,z)=\sum_{n=0}^{\infty}\frac{\tilde a_m(n)}{n!}z^n, \qquad m\in \N;\\
\tilde r_m(z)&:=\CO(\tilde R_m,z)=\sum_{n=1}^{\infty}\frac{\tilde b_m(n)}{n!}z^n, \qquad m\in \N.
\end{align*}
\begin{Lemma}\label{LtildeQR}
The exponential generating functions of the sets $\tilde Q_m$, $\tilde R_m$ have the following properties:
\begin{enumerate}
\item $\tilde a_m(n)$, $\tilde b_m (n)$ are nonnegative integers for all $m\ge 1$, $n\ge 0$.
\item $r_m(z)\preceq \tilde  r_m(z)$ and $q_m(z)\preceq \tilde  q_m(z)$ for $m\ge 1$.
\item $\tilde q_1(z)=1$; $\tilde q_2(z)=\exp z$ and
\begin{equation}\label{tilde_q}
\tilde q_m(z)=\underbrace{\exp(z\exp(z \exp(\cdots(z\exp(z\exp}\limits_{m-1 \text{\rm{ times }} \exp } z))  \cdots))),\quad m>1.
\end{equation}
\item $\tilde r_2(z)=z$ and
\begin{equation}\label{tilder}
\tilde r_m(z)=z\underbrace{\exp(z \exp(\cdots(z\exp(z\exp}\limits_{m-2 \text{\rm{ times }} \exp } (z)))  \cdots)),\quad m>1.
\end{equation}
\end{enumerate}
\end{Lemma}
\begin{proof}
By combinatorial construction above one checks by induction on $m$
that $R_m\subset \tilde R_m$, $Q_m\subset \tilde Q_m$ for $m\ge 1$, thus yielding Claims 1,2.

We compute the complexity functions for all lines~\eqref{tildeRQ}.
In the second case, $\tilde R_m$ is a product of the sets $X$ and $\tilde Q_{m-1}$ and we apply Lemma~\ref{Lprod}.
In the third case, we have a PBW-type ordering and we use Theorem~\ref{Texp}.
\begin{align*}
\tilde q_1(z)&=\CO(\tilde Q_1,z)=1;\\
\tilde r_m(z)&=\CO(\tilde R_m,z)=\CO(X,z)\CO(\tilde Q_{m-1},z)=z\tilde q_{m-1}(z), \quad m>1;\\
\tilde q_m(z)&=\CO(\tilde Q_m,z)=\exp (\CO(\tilde R_m,z))=\exp(\tilde r_m(z)), \quad m>1.
\end{align*}
Theses recurrence relations yield the formulas of Claims 3,4.
\end{proof}

\subsection{Upper bonds on complexity functions and codimension growth}
First, we get a new clear bound in terms of generating functions for Lie PI-algebras.
\begin{Th}\label{TboundL}
Let $\VV$ be a variety of Lie algebras satisfying a nontrivial identity of degree $m$ and
$\CO(\VV,z)=\sum\limits_{n=1}^\infty \frac {c_n(\VV)}{n!} z^n$ its complexity function.
The following coefficientwise bound   holds:
$$
\CO(\VV,z)\preceq \tilde r_m(z)=z\underbrace{\exp(z \exp(\cdots(z\exp(z\exp}\limits_{m-2 \text{\rm{ times }} \exp } (z)))  \cdots)).
$$
\end{Th}
\begin{proof}
Let us prove that $P_n(\VV,\{x_1,\dots,x_n\})$ is a linear span of $[w]$,
where $w$ are regular $m$-Lie indecomposable words.
The standard linearization~\cite{Ba,Drensky} yields a multilinear identity of degree $m$:
$$[X_m,X_{m-1},\dots,X_1]\equiv\sum_{e\ne\pi\in S_{m-1}}
  \lambda_\pi[X_{m},X_{\pi(m-1)},\dots,X_{\pi(1)}],\qquad \lambda_\pi\in K.$$
Let   $w\in P_n(X^*)$ be an $m$-Lie decomposable regular word and
$[w]\in P_n(\VV,\{x_1,\dots,x_n\})$ a corresponding Lie monomial given by an acceptable arrangement of brackets.
As we observed above~\eqref{m-decomp},
\begin{align*}
&w=w_m w_{m-1}\cdots w_1 \cdot b,\quad w_m>\cdots>w_1; \\
&w_t=y_t x_{t1}\cdots x_{ts_t}, \quad y_t,x_{tj}\in X,\quad
 y_t>x_{tj},\  1\le j\le s_t;\quad 1\le t\le m.
\end{align*}
Consider the subword $u=w_m w_{m-1}\cdots w_1$.
The following arrangements of brackets are acceptable
\begin{align*}
[w_t]&= [y_t,x_{t1},\dots, x_{ts_t}]=w_t+w_t^*,\qquad 1\le t\le m;\\
{[u]}&=  [[w_m],[w_{m-1}],\dots,[w_1]]=u+u^*,
\end{align*}
where $u^*$ denotes a linear combination of words smaller than $u$, we keep this notation $^*$ below.
Introduce the following regular words and acceptable arrangements of brackets
\begin{equation*}
\begin{split}
u_{\pi}  &=w_m w_{\pi(m-1)}\cdots w_{\pi(1)}<u, \\
{}[u_{\pi}]&=[[w_m],[w_{\pi(m-1)}],\dots, [w_{\pi(1)}]]=
         u_{\pi}+u_{\pi}^*,
                      \qquad\quad
\end{split}
e\ne \pi\in S_{m-1}.
\end{equation*}
Recall that $w$ and $u$ start with the maximal letter $x_n$.
Instead of $w=ub$ we take $w':=x_n b$, which is regular,
consider the left-normed arrangement of brackets on it, which is clearly acceptable, and expand it:
$[w']:=[x_nb]=x_n b+\sum_i\mu_i a_i x_n b_i$, where $\mu_i\in K$, $a_i,b_i\in X^*$.
By substitution $x_n:=[u]$,
we obtain an acceptable arrangement of brackets for the word $w$ as $[w]_1:=[[u],b]$.
It is acceptable, i.e. $\widehat{[w]_1}=w$, hence,
$[w]_1-[w]\in A(X)$ is a linear combination of words smaller than $w$.
We use the identity
\begin{align*}
[w]_1&=\big[[[w_m],\dots,[w_1]],b\big]=
 \sum_{e\ne\pi\in S_{m-1}}
  \lambda_\pi [[[w_{m}],[w_{\pi(m-1)}],\dots,[w_{\pi(1)}]],b]=\\
&=\sum_{e\ne\pi\in S_{m-1}}\lambda_\pi
   \big((u_{\pi}+u_{\pi}^*)b+   \sum_i \mu_i  a_i(u_{\pi}+u_{\pi}^*)b_i\big).
\end{align*}
We get words smaller than $w$.
Hence, $[w]\in A(X)$ is a linear combination of words smaller than $w$.
Recall, that a basis of $L(X)$ consists of $[v]=v+v^*$, where $v$ are regular words.
Therefore, $[w]\in L(X)$ is expressed via elements $[v]$,
where $v$ are regular words smaller than $w$.
Since we stay in the finite dimensional component $P_n(\VV,\{x_1,\dots,x_n\})$,  the process must stop.

We proved that, $c_n(\VV)\le b_m(n)$, $n\in\N$;
in terms of complexity functions $\CO(\VV,z)\preceq r_m(z)$.
Applying bounds of Lemma~\ref{LtildeQR}, we get $\CO(\VV,z)\preceq \tilde r_m(z)$ and use~\eqref{tilder}.
\end{proof}

Now we prove the first main result of the paper.
\begin{proof}[Proof of Theorem~\ref{TupperP}]
Let $X=\{x_i| i\in\N\}$.
It is sufficient to consider the case of the relatively free Poisson algebra $P(X):=F(\WW,X)$
defined by a unique nontrivial multilinear Lie identity of degree~$m$.
Consider the subalgebra generated by $X$ with respect to the Lie bracket $\tilde L(X)\subset P(X)$.
Let also $L(X)=L(\VV,X)$ be the relatively free Lie algebra defined by the same Lie identity and
$S(X)=S(L(X))$ the respective symmetric Poisson algebra, which is multigraded.
Clearly, we have an epimorphism $\phi: L(X)\to \tilde L(X)$.
Since the symmetric algebra is a polynomial ring generated by a basis of $L(X)$,
we get an epimorphism of commutative algebras $\bar \phi: S(X)\to P(X)$.
Let us check that $\bar\phi$ is a homomorphism with respect to the Poisson bracket as well.
We know that $\phi $ is a homomorphism of Lie algebras.
Computations below show that we get a Lie homomorphism for elements of higher degree in the symmetric algebra.
For example, let $a,b,c\in L(X)$, then
\begin{multline*}
\bar \phi(\{ab,c\})=\bar \phi (a\{b,c\}+b\{a,c\})=\phi(a)\phi(\{b,c\})+\phi(b) \phi(\{\a,c\})\\
=\phi(a)\{ \phi(b),\phi(c)\}+\phi(b) \{\phi (a),\phi(c)\}=\{\phi(a)\phi(b),\phi(c)\}
= \{\bar \phi(ab),\bar\phi(c)\}.
\end{multline*}
We use Theorem~\ref{Texp}, the bound of Theorem~\ref{TboundL}, Lemma~\ref{LtildeQR}, and~\eqref{tilde_q}
\begin{align*}
\CO(\WW,z)=\CO(P(X),z)\preceq \CO(S(X),z)=\exp \CO(L(X),z)\preceq \exp \tilde r_m(z)=\tilde q_m(z),
\end{align*}
thus yielding the first claim.

The series $\tilde q_m(z)$ has nonnegative coefficients due to~\eqref{tilde_q}.
By induction on $m$, one checks that
\begin{align*}
\lim_{r\to +\infty} \frac{ \ln^{(m-1)}  {\mathrm M}_{\tilde q_m}(r)}{r^1}=\lim_{r\to +\infty} \frac{ \ln^{(m-1)} \tilde q_m(r)}{r^1}=1.
\end{align*}
Recall that $\tilde q_m(z)=\sum_{n=0}^\infty \frac {\tilde a_m(n)}{n!}z^n$. We apply Theorem~\ref{TSher}, where  $\lambda=1$:

\begin{align*}
1&=\limsup_{r\to+\infty} {\frac {\ln^{(m-1)} {\rm M}_{\tilde q_m} (r)} {r^1}}= \limsup_{n\to\infty} \Big(\frac {\tilde a_m(n)}{n!}\Big)^{1/n} \ln^{(m-2)}n,\\
&\Big(\frac {\tilde a_m(n)}{n!}\Big)^{1/n} \ln^{(m-2)}n\le 1+o(1),\qquad n\to\infty,\\
&c_n(\WW)\le \tilde a_m(n) \le \frac{n!} {(\ln^{(m-2)}n)^n}(1+o(1))^n, \qquad n\to\infty. \qedhere
\end{align*}
\end{proof}

\section {Complexity functions of centre-by-metabelian Lie algebras}

In this section, we compute the complexity function of centre-by-metabelian Lie algebras, which is applied in the next section.
A Lie algebra is {\it centre-by-metabelian} provided that it satisfies the identity:
\begin{equation}\label{cent_metab}
[[[X_1,X_2],[X_3,X_4]],X_5]\equiv 0.
\end{equation}
The class of all such Lie algebras is a variety denoted $\CC:=[\AA^2,\EE]$~\cite[4.8.6]{Ba}.
Centre-by-metabelian groups, Lie algebras, and Lie rings were studied in more details in~\cite{Kuz77,ManSto14}.
This class proved to be important.
So, in this class a non-finitely based Lie ring without torsion was constructed~\cite{Kra09}.

\begin{Th} \label{Tcenterby}
Let $\CC$ be the variety of centre-by-metabelian Lie algebras over a field $K$. Then
\begin{enumerate}
\item $\CO(\CC,z)\preceq \frac{z^2}2 \exp(z)+z -\frac {z^3} 6$;
\item $\CO(\CC,z)= \frac{z^2}2 \exp(z)+z -\frac {z^3} 6$ for $\ch K=2$;
\item $\CO(\CC,z)= \frac{z^2}2 \exp(z)+2z -\sh(z) $ for $\ch K\ne 2$.
\end{enumerate}
\end{Th}
\begin{proof}
Let $L=F(\CC,X,K)$ be the free centre-by-metabelian Lie algebra generated by a set $X$ over a field $K$.
Then $L''=(L^2)^2$ is in the center of $L$ and $L/L''\cong F(\AA^2, X)$. Using
\cite[Lemma 3.3]{Pe95} or \cite[Theorem 3.2]{Pe99JMSciSch_exp}, we get
\begin{equation}\label{cA2}
\CO(L/L'',z)=\CO(\AA^2,z)=1+z+e^z(z-1).
\end{equation}
We use arguments~\cite[4.8.6]{Ba}.
It is well-known~\cite{Ba} that $L'$ is generated by all commutators $w=[x_{i_1},x_{i_2},\ldots,x_{i_n}]$,
$x_{i_j}\in X$, with $i_1>i_2\le i_3\le \cdots \le i_n$.
Now, $L''$ is spanned by their mutual products (see details in~\cite{ManSto14}).
Identity~\eqref{cent_metab} implies $[[y_1,y_2],[y_3,y_4,y_5]]=-[[y_1,y_2,y_5],[y_3,y_4]]$.
We move almost all variables into one commutator and
put the largest variable deep inside that longest commutator. So,   $L''$ is spanned by
\begin{equation}
\begin{split}
&[[x_{i_1},x_{i_2}],[x_{i_3},x_{i_4},\ldots,x_{i_n}]], \\
& i_1>i_2,\quad i_4\le \cdots\le  i_n, \quad i_3=\max\{i_1,\ldots,i_n\}.
\end{split}
\end{equation}
Let us estimate the number of such multilinear monomials. We have $i_3=n$ and choose a $\{i_1,i_2\}$ among $x_1,\ldots,x_{n-1}$.
Hence, $c_n(L'')\le \frac {(n-1)(n-2)}2=d_n$ for $n\ge 4$ and
\begin{multline*}
\CO(L'', z)
 \preceq \sum _{n\ge 4}\frac {(n-1)(n-2)}2 \frac {z^n}{n!}
=\sum _{n\ge 4}\Big (\frac {n(n-1)}2- n+1\Big )\frac {z^n}{n!}
=\frac{z^2}2\sum_{n\ge 2}\frac {z^n}{n!}-z\sum_{n\ge 3}\frac {z^n}{n!}+\sum _{n\ge 4}\frac {z^n}{n!} \\
=\frac{z^2}2\Big (e^z-1-z\Big )
-z\Big (e^z-1-z-\frac {z^2}2\Big)+ \Big (e^z-1-z-\frac {z^2}2 -\frac {z^3}6\Big )
=e^z\Big (\frac {z^2}2-z+1\Big)-1-\frac {z^3} 6.
\end{multline*}
Adding~\eqref{cA2}, we get the first claim.

Now, let $L=F(\CC,X,\Z)$ be the free centre-by-metabelian Lie ring.
Kuzmin proved that $L''$ is a direct sum of a free abelian group and an elementary 2-group~\cite{Kuz77}.
Later, his computations of both the free component and the torsion subgroup were corrected in~\cite{ManSto14} and~\cite{KovSto14}.
Following~\cite{ManSto14}, the monomials
\begin{equation}\label{xi1}
[[x_{i_1},x_{i_2}],[x_{i_3},x_{i_4},x_{i_5},\ldots, x_{i_n}]], \quad x_{i_j}\in X,\quad n\ge 5;\\
\end{equation}
are called {\it Kuzmin elements} if
$i_1>i_2$, $i_3>i_4$, $i_1\ge i_3$,  and $i_4\le i_2\le i_5\le i_6\le \cdots \le i_n$.
Building on the results of~\cite{ManSto14} it was shown in~\cite[Theorem 8.1]{KovSto14}
that the multilinear component $P_n(L'')$ of even degree $n$ is a free abelian group that is freely generated
by the multilinear Kuzmin elements of degree $n$ and the element
$$[[x_3,x_2],[x_4,x_1,x_5, x_6,\ldots, x_{n}]],\quad n\ge 6,$$
while in the case of odd degree $n\ge 5$, the multilinear component $P_n(L'')$ is a direct sum
of a free abelian group that is freely generated by the multilinear Kuzmin elements of degree $n$ and a cyclic
group of order 2 that is generated by the element
$$
[[x_1,x_2],[x_3,x_4,\ldots,x_n]]+ [[x_2,x_3],[x_1,x_4,\ldots,x_n]]+ [[x_3,x_1],[x_2,x_4,\ldots,x_n]].
$$
The component of degree 4 was not considered in~\cite{Kuz77,ManSto14}.
The identity~\eqref{cent_metab} is of degree 5, so $P_4(L'')$ is the same as that for the free Lie ring, namely
$\langle [[x_{\pi(1)},x_{\pi(2)} ],[x_{\pi(3)},x_{\pi(4)}]] \mid \pi\in S_4\rangle_\mathbb Z$,
which has rank 3. Its basis corresponds to two Kuzmin elements~\eqref{xi1} and the additional element above.
Thus, the description above applies to the degree 4 as well.


Now consider the case of the algebra over a field $K$.
The number of multilinear Kuzmin elements~\eqref{xi1} is equal to $\frac {n(n-3)}2$, see~\cite[Lemma 7.1, Theorem 7.1]{ManSto14}.
In case $\ch K=2$, by arguments above each multilinear component obtains one more element,
we get the same formula for all integers (this fact follows also from~\cite{KovSto14})
$c_n(L'')=1+\frac {n(n-3)}2=\frac {(n-1)(n-2)}2=d_n$ for $n\ge 4$, the same value as above, yielding the second claim.
Let $\ch K\ne 2$, then the 2-torsion elements disappear, the latter having the complexity function
\begin{equation*}
\sum_{k\ge 2} \frac{z^{2k+1}}{(2k+1)!}=\sh(z)-z-\frac {z^3}{6}.
\end{equation*}
It remains to subtract this value from the function of claim 2.
(These values  $c_n(L'')=\frac{n(n-3)}2$ or $c_n(L'')=\frac{(n-1)(n-2)}2$ for odd and even $n\ge 4$
are also shown in~\cite[4.8.6]{Ba} in case of $\C$).
\end{proof}

\section{Codimension growth of solvable Poisson algebras}

Now we show that there exist Poisson PI-algebras, satisfying some Lie identical relations,
that  correspond to the functions of our scale for the codimension growth~\eqref{scale}.
We shall use the following recent result of S.Siciliano and H.Usefi.
\begin{Th}[{\cite[Theorem 3.3]{SicUse21}}]\label{TSicUse}
Let $P$ be a Poisson algebra that is Lie solvable of length~$n$.
Then the Poisson ideal generated by all elements
$\{\{\{x_1,x_2\},\{x_3,x_4\}\},x_ 5\}$ is
associative nilpotent of index bounded by a function of~$n$.
\end{Th}

\begin{proof}[Proof of Theorem~\ref{TPpoly}]
We start with a general remark. 
Denote by $X$ a countable set.
Let $\VV$ be a multilinear variety of Lie algebras. Consider the free Poisson algebra
$F(X)=S(L(X))$ and the relatively free Poisson algebra,  defined by all Lie identities of $\VV$,
as the factor algebra by the verbal ideal $F(\PP\VV,X)=F(S(L(X)))/\VV(F(S(L(X))))$,
see definitions and properties in~\cite{Ba,Drensky}).
Let $h(X_1,\ldots,X_m)\in\VV$ be a multilinear Lie identity.
By substitutions of Lie elements $X_i=a_i\in L(X)$, $i=1,\ldots,m$ we get $h(a_1,\ldots,a_m)\in \VV(L(X))$.
If at least one of the elements above is a product $a_1=b_1\cdots b_k$, $b_i\in L(X)$, where $k\ge 2$,
then by Leibnitz rule, $h(b_1\cdots b_k,a_2,\ldots,a_m)\in (L(X))^k S(L(X))$.
Let $\tilde L(X)\subset F(\PP\VV,X)$ be the subalgebra generated by $X$ using Lie bracket only, then 
by our arguments $\tilde L(X)\cong F(\VV,X)$, the relatively free algebra of the variety of Lie algebras $\VV$.

Now let $P:=F(\PP\NN_{s_q}\cdots\NN_{s_1},X)$ be
the relatively free Poisson algebra defined by
the identity of Lie polynilpotency corresponding to the tuple $(s_q,\ldots,s_1)$.
Consider the Lie subalgebra $L$ generated by $X$. By arguments above, 
$L\cong F(\NN_{s_q}\cdots\NN_{s_1},X)\subset P$.
Consider the verbal ideal determined by the centre-by-metabelian identity
(i.e. the ideal generated by all substitutions of elements of $L$ into identity~\eqref{cent_metab}). 
We obtain a multihomogeneous ideal $H:=\CC(L)\triangleleft L$.
Choose well-ordered multihomogeneous bases $\{v_i| i\in I\}$  for $L/H$ and $\{ w_j| j\in J\}$ for $H$.
The following monomials span the Poisson algebra $P$:
\begin{equation}\label{wwww}
P=\langle v_{i_1}\cdots v_{i_a} w_{j_1}\cdots w_{j_b}\mid  i_s\in I, j_s\in J,
i_1\le \cdots\le  i_a, j_1\le  \cdots \le j_b,\ a,b\ge 0 \rangle_K.
\end{equation}
To evaluate the number of the first factors above remark that $L/H\in \CC$, so
$\CO(L/H,z)\preceq \CO(\CC,z)$, and use Theorem~\ref{Texp}.
Denote $f(z)=\CO(L,z)=\CO(\NN_{s_q}\cdots\NN_{s_1},z)$.
Since $L$ is solvable, by Theorem~\ref{TSicUse}, there exists $N$ such that $b\le N$ in~\eqref{wwww}.
Using machinery of exponential generating functions (Lemma~\ref{Lprod}),
and  the  bound on the number of the seconds factors~\eqref{wwww}, we get
\begin{equation}\label{coP}
\CO(P,z)\preceq \exp(\CO(L/H,z)) (\CO(H,z))^N\preceq \exp(\CO(\CC,z)) (\CO(L,z))^N
\preceq \exp(\CO(\CC,z)) (f(z))^N .
\end{equation}
Below, we substitute $z=r\in\R^+$ into the series with nonnegative coefficients, turning $\preceq$ into $\le $.
By Theorem~\ref{TordNN}, we have a bound
\begin{equation}\label{frr}
f(r)\le \exp^{(q-1)}\Big(\Big(\frac {s_2}{s_1}+o(1)\Big)r^{s_1} \Big),\quad r\to+\infty.
\end{equation}
By~\eqref{coP}, \eqref{frr}, Theorem~\ref{Tcenterby}, and using that $q\ge 3$, we get
\begin{align*}
\ln\CO(P,r)&\le \CO(\CC,r)+N \ln f(r)\le \frac {r^2}2 \exp(r)+r+N \exp^{(q-2)}\Big(\Big(\frac {s_2}{s_1}+o(1)\Big)r^{s_1} \Big)\\
&\le r^2\exp^{(q-2)}\Big(\Big(\frac {s_2}{s_1}+o(1)\Big)r^{s_1} \Big),\qquad r\to+\infty;\\
&\lim_{r\to+\infty}
 \frac{\ln^{(q-1)}{\CO(P,r)}}{r^{s_1}}=\frac{s_2}{s_1}.
\end{align*}
Recall that $\CO(P,z)=\sum_{n=0}^\infty \frac {c_n(P)}{n!}z^n$. Applying Theorem~\ref{TSher}, where  $\lambda=s_1$, we get

\begin{align*}
\frac{s_2}{s_1}&=\limsup_{r\to+\infty} {\frac {\ln^{(q-1)} {\CO(P,r)}} {r^{s_1}}}
= \limsup_{n\to\infty} \Big(\frac {c_n(P)}{n!}\Big)^{s_1/n} \ln^{(q-2)}n;\\
&\Big(\frac {c_n(P)}{n!}\Big)^{s_1/n} \ln^{(q-2)}n\le \frac{s_2}{s_1}+o(1),\qquad n\to\infty;\\
c_n(P)&\le
  \displaystyle
  \frac{n!}{(\ln^{(q-2)}n)^{n/s_1}}
                  \Big(\frac{s_2+o(1)}{s_1}
                  \Big)^{n/s_1},\qquad  n\to\infty.
\end{align*}
Finally, the lower bound follows by $c_n(\PP\NN_{s_q}\cdots\NN_{s_1})=c_n(P)\ge c_n(L)=c_n(\NN_{s_q}\cdots\NN_{s_1})$
and the lower bound of Theorem~\ref{Tpoly}.
\end{proof}

\section{Wild codimension growth in case of mixed identities}

In this section we show that
the codimension growth of Poisson PI-algebras without Lie identities can be very high.
Recall that~\eqref{cofrrP}, the codimension growth of the free Poisson algebra is $c_n(\PP)=n!$.
Unlike varieties with Lie identities (see our Theorem~\ref{TupperP}),
there exist Poisson algebras with mixed identities
such that $c_n(\WW_s)\ge (n-1)!$, see below.

Let us compute the complexity function of the
variety of Poisson algebras $\WW_s$,  defined by the mixed identity:
\begin{equation}\label{xxxS}
\{X_1,X_2\}\cdots \{X_{2s-1},X_{2s}\}\equiv 0, \quad s=2,3,\ldots.
\end{equation}
\begin{Th} \label{TWs}
The complexity function of the variety of Poisson algebras $\WW_s$ is equal to
\begin{enumerate}
\item
$\displaystyle \CO(\WW_s)=\exp(z)\bigg(1+\sum_{k=1}^{s-1} \frac 1{k!}\big(-z-\ln(1-z)\big)^k \bigg).$
\item
$\displaystyle \CO(\WW_s)= \frac1{1-z}- \exp(z)\bigg(\sum_{k=s}^{\infty}  \frac 1{k!} \Big(\sum_{n=2}^\infty \frac{z^n}n\Big)^k \bigg).$
\end{enumerate}
\end{Th}
\begin{proof} Let $R$ be a Hall basis of the free Lie algebra $L(X)\subset F(X)=S(L(X))$ and
$R=\cup_{n\ge 1} R_n$, where $R_n$ are monomials of length $n$ in the generators $X$.
In particular, $R_1=X$, and $\CO(R_1,z)=z$.
Consider the commutators subalgebra: $L'=[L,L]=\langle R_n| n\ge 2\rangle _K$.
Using~\eqref{complL}, we get a function with non-negative coefficients:
\begin{equation}\label{hz}
h(z):=\CO(L',z)=\CO(L,z)-\CO(X,z)=\sum_{n=2}^\infty \frac{z^n}n=-z-\ln(1-z).
\end{equation}

The free Poisson algebra $F(X)$ has the PBW-basis
\begin{equation}\label{free}
x_{i_1}\cdots x_{i_a}\cdot w_{j_1}\cdots w_{j_b},\qquad x_{i_j}\in X,\quad  w_{i_j}\in \mathop{\cup}\limits_{n\ge 2} R_n,
\end{equation}
where all multiplicands above are nondecreasing with respect to a liner order on $R(X)$.
We consider multilinear elements above.

We claim that the verbal ideal $\WW_s(F(X))$ is spanned by monomials~\eqref{free} with $b\ge s$.
First, clearly, all such products belong to the verbal ideal.
Second, consider a substitution into $\{X_1,X_2\}$ of products~\eqref{free}.
By Leibnitz rule we get a sum of products each containing a factor $\{w_i,w_j\}$, where $w_i,w_j\in R$,
thus yielding a factor $w_l\in \cup_{n\ge 2} R_2$.
So, the verbal ideal contains at least $s$  such factors.
Finally, a basis of $F(\WW_s,X)\cong F(X)/\WW_s(F(X))$ consists of monomials~\eqref{free} with $b=0,1,\dots,s-1$.

Consider a formal set of nonordered(!) products of length $k$ and use Lemma~\ref{Lprod} and~\eqref{hz}:
\begin{align*}
U_k:&=\{w_{i_1}\cdots w_{i_k}\mid w_{i_j}\in \mathop{\cup}_{n\ge 2} R_n\};\\
\CO(U_k,z)&=\CO(\cup_{n\ge 2} R_n, z)^k=\CO(L',z)^k=h(z)^k.
\end{align*}
Now we pass to ordered multilinear products:
\begin{align}\nonumber
\tilde U_k:&=\{w_{i_1}\cdots w_{i_k}\mid w_{i_j}\in \cup_{n\ge 2} R_n, \quad i_1<\cdots< i_k\};\\
\CO(\tilde U_k,z)&= \frac{1}{k!} \CO(U_k,z) =\frac{1}{k!} h(z)^k.
\label{coUk}
\end{align}
By~\eqref{commutative}, the first factor~\eqref{free} has the complexity function $\exp(z)$.
We count the complexity function of the second factor in~\eqref{free} summing~\eqref{coUk} for $k=0,\ldots,s-1$
and using~\eqref{hz}, thus yielding the first claim.

Finally,
\begin{align*}
\CO(\WW_s),z&=\exp(z)\Big (1+\sum_{k=1}^{s-1} \frac  {h^k(z)}{k!} \Big)=
\exp(z)\Big (\sum_{k=0}^{\infty } \frac  {h^k(z)}{k!}- \sum_{k=s}^{\infty } \frac  {h^k(z)}{k!}\Big)\\
&=\exp(z) \Big(\exp (h(z))- \sum_{k=s}^{\infty } \frac  {h^k(z)}{k!}\Big)
=\frac 1{1-z}-\exp(z)\bigg(\sum_{k=s}^{\infty}  \frac 1{k!} \Big(\sum_{n=2}^\infty \frac{z^n}n\Big)^k \bigg).\qedhere
\end{align*}
\end{proof}

{\it Examples.} We provide some computations of the series:
\begin{align*}
\CO(\WW_2,z)&=1+z+z^2+z^3+\frac 78 z^4+\frac{17}{24}z^5+\frac{41}{72}z^6+\frac{169}{360}z^7
+\frac{51}{128}z^8+ \frac{25133}{72576}z^9 + \frac{556037}{1814400}z^{10}+\cdots\\
\CO(\WW_3,z)&=1+z+z^2+z^3+z^4+z^5+\frac{47}{48}z^6+\frac{15}{16}z^7+\frac{1021}{1152}z^8+\frac{43249}{51840}z^9+\frac{509}{648}z^{10}+\cdots\\
\CO(\WW_4,z)&=1+z+z^2+z^3+z^4+z^5+z^6+z^7+\frac{383}{384}z^8+\frac{1141}{1152}z^9+\frac{5641}{5760}z^{10}+\cdots\\
\CO(\WW_5,z)&=1+z+z^2+z^3+z^4+z^5+z^6+z^7+z^8+z^9+  \frac{3839}{3840}z^{10}+\cdots
\end{align*}

We obtain a version of the bound of S.Ratseev~\cite[Theorem 1]{Rats14}.
\begin{Corr}
$c_n(\WW_s)\ge [e(n-1)!]-1$ for $s\ge 2$ , where $e=2,71828...$.
\end{Corr}
\begin{proof}
By Theorem,
\begin{multline*}
\CO(\WW_s,z)\succeq \CO(\WW_2,z) = \exp(z)(1-z-\ln(1-z))
=\Big(\sum_{n\ge 0} \frac {z^n} {n!}\Big )\Big(1+\sum_{n\ge 2}\frac {z^n}n\Big )\\
=1+z+\sum_{n=2}^\infty z^n\Big (\frac 1{n\cdot 0!}+ \frac 1{(n-1)1!}+\cdots+ \frac 1{2(n-2)!}  +\frac 1{n!}\Big)\\
\succeq 1+z+\sum_{n=2}^\infty \frac{z^n} n\Big(\frac 1{0!}+\frac 1{1!}+\cdots +\frac 1{(n-2)!}+ \frac 1{n!}\Big).
\end{multline*}
Hence,
$c_n(\WW_s)\ge (n-1)!(e-\frac 1{(n-1)!}-\sum_{k> n}\frac 1{k!})\ge
[(n-1)!e]-1$.
\end{proof}



\end{document}